\documentclass[12pt]{amsart}
\usepackage{amsmath,amsfonts,amssymb,amsthm,ifpdf}
\usepackage{graphicx,color}
\ifpdf
\usepackage[pdftex,pdfstartview=FitH,pdfborderstyle={/S/B/W 1}, %
colorlinks=true, linkcolor=blue, urlcolor=blue,citecolor=blue,%
pagebackref=true %
]{hyperref}

\makeatletter
\newcommand\fake@math{}
\def\fake@math#1\){[math]}
\makeatother

\else
\usepackage[hypertex,colorlinks]{hyperref}
\fi

\DeclareMathRadical{\sqrtsign}{symbols}{"70}{largesymbols}{"70}
\newcommand{\bb}{\mathbb}

\newcommand{\cx}{{\bb C}}

\newcommand{\integers}{{\bb Z}}
\newcommand{\natls}{{\bb N}}

\newcommand{\reals}{{\bb R}}
\newcommand{\proj}{{\bb P}}
        
\newlength{\figboxwidth}             
\renewcommand{\bold}[1]{\medskip \noindent {\bf #1 }\nopagebreak}

\newcommand{\dirsum}{\oplus}

\newcommand{\cross}{\times}

\newcommand{\st}{\;\: : \;\:}         

\newcommand{\zed}{\integers}

\newcommand{\SL}{\operatorname{SL}}

\def\@ifundefined#1#2#3%
  {\expandafter\ifx\csname#1\endcsname\relax#2\else#3\fi}

\@ifundefined{theoremstyle}{
}{
\theoremstyle{plain} 
}
\newtheorem{theorem}{Theorem}[section]

\newtheorem{proposition}[theorem]{Proposition}
\newtheorem{lemma}[theorem]{Lemma}

\newtheorem{corollary}[theorem]{Corollary}
\newtheorem{claim}[theorem]{Claim}

\@ifundefined{theoremstyle}{
}{
\theoremstyle{definition} 
}
\newtheorem{definition}[theorem]{Definition}

\newtheorem{remark}[theorem]{Remark}

\newcommand{\cB}{{\mathcal B}}
\newcommand{\cC}{{\mathcal C}}

\newcommand{\cF}{{\mathcal F}}
\newcommand{\cG}{{\mathcal G}}
\newcommand{\cH}{{\mathcal H}}

\newcommand{\cM}{{\mathcal M}}
\newcommand{\cN}{{\mathcal N}}

\newcommand{\cP}{{\mathcal P}}

\newcommand{\cW}{{\mathcal W}}

\mathchardef\GG="321D
\newcommand{\bE}{{\mathbf E}}
\newcommand{\bF}{{\mathbf F}}

\newcommand{\bH}{{\mathbf H}}

\newcommand{\bV}{{\mathbf V}}
\newcommand{\bW}{{\mathbf W}}

\newcommand{\bfv}{{\mathbf v}}
\newcommand{\bfw}{{\mathbf w}}

\newcommand{\bbC}{{\mathbb C}}

\newcommand{\bbH}{{\mathbb H}}

\newcommand{\bbP}{{\mathbb P}}

\newcommand{\bbR}{{\mathbb R}}

\newcommand{\mcc}[1]{{}}

\title[Invariant measures for Projective Cocycles]
{Projective cocycles over $SL(2,\bbR)$ actions: measures invariant under the upper triangular group. }
\author{Christian Bonatti}
\author{Alex Eskin}
\author{Amie Wilkinson}

\begin{document}
\begin{abstract} We consider the action of $SL(2,\reals)$ on a vector
  bundle $\bH$ preserving an ergodic probability measure $\nu$ on the base $X$.   Under an irreducibility assumption on this action, we prove that if $\hat\nu$ is any lift of $\nu$ to a probability measure on the projectivized bunde $\proj(\bH)$ that is invariant under the upper triangular subgroup, then $\hat \nu$ is  supported in the projectivization $\proj(\bE_1)$ of the top Lyapunov subspace of the positive diagonal semigroup.  We derive two applications. First, the Lyapunov exponents for the Kontsevich-Zorich cocycle depend continuously on affine measures, answering a question in \cite{Matheus:Moeller:Yoccoz:Criterion}. Second, 
if $\proj(\bV)$ is an irreducible,  flat  projective bundle over a compact hyperbolic surface $\Sigma$, with hyperbolic foliation $\cF$ tangent to the flat connection, then the foliated horocycle flow on $T^1\cF$ is uniquely ergodic if the top Lyapunov exponent of the foliated geodesic flow is simple.  This generalizes results in \cite{Bonati:Comez-Mont} to arbitrary dimension.
\end{abstract}
\maketitle

\centerline{\em To the memory of Jean-Christophe.}

\section{Introduction}
\label{sec:intro}

Let $G$ denote the group $SL(2,\reals)$, and for $t\in\reals$, 
let
\begin{displaymath}
a^t = \begin{pmatrix} e^t & 0 \\ 0 & e^{-t} \end{pmatrix}, \qquad
u^t_+ = \begin{pmatrix} 1 &t \\ 0 & 1 \end{pmatrix}, \quad \hbox{and }\,
u^t_- = \begin{pmatrix} 1 & 0 \\ t & 1 \end{pmatrix}.
\end{displaymath}
Let $A = \{ a^t \st t \in \reals \}$, $U_+ = \{ u_+^t \st t \in \reals \}$, and $U_- = \{ u_-^t \st t \in \reals \}$.  
The group $G$ is generated by $A,U_+$ and $U_-$. 
We denote  by $P = A U_+$ the group of upper triangular matrices, and we 
recall that the group $P$ is solvable and hence \emph{amenable}:  any action of $P$ on a compact metric space 
admits an invariant probability measure (see, e.g. \cite[Section 8.4]{EinsiedlerWard}).  Recall that $PSL(2,\bbR)$ admits a natural identification
with the unit tangent bundle $T^1\bbH^2$ where $\bbH^2$ is the upper half plane endowed with the hyperbolic metric.  Through this identification, the left action of $SL(2,\bbR)$ on itself
induces an identification of $A=\{a^t\}$ with the geodesic flow, and $U_+$ (resp. $U_-$) with the unstable (resp. stable) horocycle flow of $\bbH^2$. 

In this paper, we consider the following situation.  Suppose $G$ acts
on a separable metric space $X$, preserving an ergodic Borel
probability measure $\nu$.  By the Mautner phenomenon (see e.g.\
\cite{Mozes:epimorphic}), $\nu$ is ergodic
with respect to the action of the diagonal group $A$. 
Let $\bH \to X$ be a continuous vector bundle over $X$ with fiber a finite dimensional vector space $H$,
and write  $\bH(x)$ for the fiber of $\bH$ over $x\in X$. 
Suppose that $G$ acts on $\bH$ by linear automorphisms on the fibers and
the given action on the base.
We denote the action of $g\in G$ on $\bH$ by $g_\ast$, so for $\bfv \in \bH(x)$, we have $g_\ast \bfv \in
\bH(gx)$.

{  Assume that $\bH$ is equipped with a Finsler structure (that is, a continuous choice of norm $\{\|\cdot\|_x : x\in X\}$ on the fibers of $\bH$), and
that with respect to this Finsler, the action of $G$ on $\bH$ satisfies the 
 following integrability condition:
\begin{equation}
\label{eq:cocycle:integrability}
\int_X \sup_{t\in [-1,1]} \left(\log\| a^t_\ast \|_x \right) \, d\nu(x) < +\infty,
\end{equation}
where, for $g\in G$,   $\| g_\ast \|_x$ denotes the operator norm of the linear action of $g_\ast$  on the fiber over $x$ with respect to the Finsler structure:
\[
\|g_\ast  \|_x  : = \sup_{\{\bfv\in H: \|\bfv\|_x =1\}} \| g_\ast \bfv\|_{gx}.
\]
}

Since the $A$-action on $X$ is ergodic with respect to $\nu$,  the Oseledets multiplicative ergodic theorem implies that  there exists 
a $a^t$-equivariant splitting
\begin{equation}
\label{eq:osceledts:two:sided:splitting}
\bH(x) = \bigoplus_{j=1}^m \bE_j(x),
\end{equation}
defined for $\nu$-almost
every $x \in X$,
and real numbers $\lambda_1 > \dots > \lambda_m$ (called the {\em Lyapunov
exponents} of the $A$-action) such that for $\nu$-a.e. $x\in X$ and all $\bfv \in \bE_j(x)$,
\begin{displaymath}
\lim_{|t| \to \infty} \frac{1}{t} \log \frac{ \| a^t_\ast  \bfv \|_{a^tx}}{\|\bfv\|_x} =
\lambda_j. 
\end{displaymath}

\begin{definition}[$\nu$-measurable invariant subbbundle]
\label{def:almost:invariant:splitting}
A $\nu$-measurable invariant subbundle for the $G$ action on $\bH$ is a measurable linear subbundle $\bW$   of $\bH$ with the property that $g_\ast \left( \bW(x) \right)= \bW(g x)$, 
for $\nu$-a.e. $x\in X$ and every $g\in G$.
\end{definition}

\begin{definition}[irreducible]
\label{def:irreducible}
We say that  the $G$ action on $\bH$  is {\em irreducible with respect to the $G$-invariant measure $\nu$ on $X$} 
if it does not admit $\nu$-measurable invariant subbundles: that is, if $\bW$ is a $\nu$-measurable invariant subbundle, then  either $\bW(x) = \{0\}$, $\nu$-a.e.  or  $\bW(x) = \bH(x)$, $\nu$-a.e.
\end{definition}
We note that the bundles $\bE_j$ in the splitting (\ref{eq:osceledts:two:sided:splitting})
are not invariant in the sense of
Definition~\ref{def:almost:invariant:splitting}, since they are
(in general) equivariant only under  the subgroup $A$ and not under all of $G$.

Let $\proj(H)$ be the projective space of $H$ (i.e., the space of
lines in $H$), and let   $\bbP(\bH)$ be the  projective bundle associated to $\bH$ with fiber $\bbP(H)$.
Then $G$ also acts on $\proj(\bH)$ via the induced projective action on the fibers. The space $ \proj(\bH)$ may
not support a $G$-invariant measure, but since $P$ is amenable and
$\proj(H)$ is compact, it
will always support a $P$-invariant measure. 
In particular, for any
$P$-invariant measure $\mu$ on $X$, there will be a
$P$-invariant measure $\hat{\mu}$ on $\proj(\bH)$ that projects to
$\mu$ under the natural map $\proj(\bH) \to X$. For such a $\mu$, denote by $ \cM^1_P(\mu)$ the (nonempty) set of all  $P$-invariant Borel probability measures on $\proj(\bH)$ 
projecting to $\mu$ on $X$.

We can now state our main theorem.
\begin{theorem}
\label{theorem:P:uniq:ergodicity}
 Suppose that the $G$-action on $\pi\colon \bH\to X$ is irreducible with respect to the $G$-invariant (and therefore $P$-invariant) measure $\nu$ on $X$, and let $\hat\nu\in  \cM^1_P(\nu)$.   Disintegrating $\hat\nu$ along the fibers of $\bbP(\bH)$, write
\begin{displaymath}
d\hat{\nu}([\bfv]) =  d\eta_{\pi(\bfv)}([\bfv])\,  d\nu(\pi(\bfv)),
\end{displaymath}
where $[\bfv]\in \proj(\bH)$ denotes the line determined by $\{0\} \neq \bfv\in \bH$.
Then for $\nu$-a.e. $x\in X$, the measure
$\eta_x$ on $\proj(\bH)(x)$ is supported 
on $\proj(\bE_1(x))$, where as in
(\ref{eq:osceledts:two:sided:splitting}), $\bE_1(x)$ is the Lyapunov
subspace corresponding to the top Lyapunov exponent of the $A$-action.

In
particular, if ${\bE}_1$ is one-dimensional, then $\# \cM^1_P(\nu) = 1$. 
\end{theorem}

The same conclusions of Theorem~\ref{theorem:P:uniq:ergodicity}  hold
when $G=SL(2,\reals)$ is replaced by any rank 1 semisimple Lie group,
and $P$ denotes the minimal parabolic subgroup of
$G$ (i.e.\ the normalizer of the unipotent radical of $G$).   

$P$-invariant measures are a natural object of study, as they are closely related to stationary measures of a $G$-action.  
Suppose that $G$ acts on a space $\Omega$, and let  $m$ be a Borel probability measure on $G$.  
Recall that a probability measure $\rho$ on $\Omega$ is  {\em $m$-stationary} if $m\ast \rho = \rho$, where for $A\subset \Omega$ measurable, we define
\[m\ast\rho (A) = \int_{G} \rho(g A) \, dm(g).
\]
A  compactly supported Borel probability measure $m$ on $G$ is {\em
  admissible} if the following two conditions hold: first, there
exists a $k\geq 1$ such that the $k$-fold convolution $m^{\star k}$ is
absolutely continuous with respect  to Haar measure; second,
$\hbox{supp}(m)$ generates $G$ as a semigroup.
Furstenberg  (\cite{F1}, \cite{F2}, restated
as \cite[Theorem 1.4]{Nevo:Zimmer}),
proved that there is a  1--1 correspondence between
$P$-invariant measures on $\Omega$  and $m$-stationary measures for
admissible $m$. In fact any $m$-stationary measure on $\Omega$ is of the form
$\lambda \ast \hat{\nu}$ where $\lambda$ is the unique $m$-stationary
measure on the Furstenberg boundary of $G$ (which is in our case the
circle $G/P$) and $\hat{\nu}$ is a $P$-invariant measure on $\Omega$.  

\bold{Remark.} 
In light of the discussion above, for any admissible measure $m$ on
$G$, Theorem~\ref{theorem:P:uniq:ergodicity} also gives a classification of
the $m$-stationary measures on ${\bH}$ projecting to $\nu$.
\medskip

In the context where $X = SL(2,\reals)/\Gamma$, with $\Gamma$ cocompact (or of finite covolume), $\bH$ a flat $H$-bundle over $X$, and $H = \reals^2$, $\cx^2$ or $\cx^3$, 
Theorem~\ref{theorem:P:uniq:ergodicity} was proved by Bonatti and Gomez-Mont \cite{Bonati:Comez-Mont}, where irreducibility is replaced with the equivalent  
hypothesis that $\rho(\Gamma)$ is Zariski dense, where $\rho$ is the monodromy representation.

Some constructions used in the proof of Theorem~\ref{theorem:P:uniq:ergodicity} are also used in \cite{EM} 
in their classification of $SL(2,\reals)$ invariant probability measures on moduli spaces.

\section{Applications and the irreducibility criterion}

The irreducibility hypothesis in Theorem~\ref{theorem:P:uniq:ergodicity}  is not innocuous.  Checking for the non-existence of invariant {\em measurable} subbundles is in general an impossible task, but there are two restricted contexts where it is feasible, on which we focus here:
\begin{itemize}
\item{Suppose that $X = G/\Gamma$, for some discrete
    subgroup $\Gamma \subset G$, and $\bH$ is a flat bundle over
    $X = G/\Gamma$ with monodromy representation
    $\rho\colon \Gamma\to GL(H)$.  Then $G$ acts transitively on $X$,
    the $P$-invariant measures on $X$ are all algebraic by
    \cite{Mozes:epimorphic} (which uses Ratner's Theorem), and
    irreducibility is then equivalent to the condition that there are
    no invariant algebraic subbundles of $\bH$.  In the case where
    $\nu$ is Haar measure, irreducibility of the associated $G$-action
    reduces to the condition that $\rho$ is an irreducible
    representation. In Subsection~\ref{ss=monodromy} we derive some
    consequences of Theorem~\ref{theorem:P:uniq:ergodicity} in this
    context.}
\item
If the bundle $\bH$ admits a Hodge structure (not necessarily
  $G$-invariant) then checking irreducibility can sometimes be reduced
  to showing that there are no invariant subbundles that are
  compatible with the Hodge structure, a much simpler task (since such
  subbundles must be real-analytic).  {In particular, the condition that $G$ acts transitively on 
the base  in the previous setting can be relaxed.}
This has been established rigorously by
  Simion Filip for the Kontsevich-Zorich action, and we use
  Theorem~\ref{theorem:P:uniq:ergodicity} in Subsection~\ref{ss=KZ} to
  deduce further results in that context.  

  The mantra here is that for
  such Hodge bundles whose base supports a $G$-invariant measure, any
{measurable} 
  $G$-invariant subbundle must come from algebraic geometry. For the
  Kontsevich-Zorich action, this has been established in \cite{EFW}.

\end{itemize}

We now describe the applications in more detail.

\subsection{Linear representations of \(G\)-lattices}\label{ss=monodromy}
$  $

The first application concerns the dynamics of the horocycle flow of a foliation with hyperbolic leaves.  If a manifold
$M$ has a 2-dimensional foliation $\cF$ whose leaves carry continuously varying hyperbolic structures, then there is a natural $G$-action on the unit  tangent bundle $T^1\cF\subset T^1M$ to the leaves, induced by the identification of $T^1{\bbH}^2$ with $PSL(2,\reals)$ mentioned in the introduction.  This gives rise to the foliated geodesic and horocyclic flows on $T^1\cF$.  

The
dynamical properties of these flows have geometric consequences for the foliations, including properties of harmonic measures (A probability measure on $M$ is  {\em harmonic} with respect to the foliation $\cF$ if it is invariant under leafwise heat flow: see \cite{Garnett}).    In particular, a probability measure on $M$ is harmonic along the leaves of $\cF$ if and only if it is the projection of a $P$-invariant measure on $T^1\cF$ (\cite{Martinez, BakhtinMartinez}).  We consider here the special case of foliations induced by representations of surface groups into linear groups.   First, we discuss the suspension construction.

\subsubsection{The suspension construction and a criterion for simplicity}
$ $

Let $\Gamma < SL(2,\reals)$ be a  lattice,  and let $\rho\colon \Gamma\to GL(H)$ be a representation, where $H=\reals^k$ or $\bbC^k$, $k\geq 2$.
Then $\Gamma$ acts on $SL(2,\reals)\times H$ by the diagonal action
$$
\gamma\colon (g, v)\mapsto ( g\gamma^{-1} ,  \rho(\gamma) v).
$$
The group $G=SL(2,\reals)$ also acts on $SL(2,\reals)\times H$, by left multiplication in the first factor, and trivially in the second factor.  The $\Gamma$ and $G$ actions commute.

Define the  {\em suspension} of $\rho$:
\[\bH_\rho = SL(2,\reals)\times H/\Gamma;\] it is a $H$-bundle over $SL(2,\reals)/\Gamma$ admitting the quotient (left) $G$-action.   The orbits of this $G$ action foliate $\bH_\rho$; this foliation is the quotient of the horizontal foliation  of $SL(2,\reals)\times H$ by $\Gamma$.  

The  suspension construction is quite general: if $\Gamma$ is the covering group of a normal cover $\hat M$ of a manifold $M$, and  $\rho\colon \Gamma\to Aut(X)$ is  a homomorphism into the automorphism group of an object $X$, then there is an associated fiber bundle $X_\rho \to M$ with fiber $X$ obtained by taking the quotient of $\hat M\times X$ by the  diagonal action of $\Gamma$, acting in the first factor by deck translations and in the second factor by $\rho$.  If a group $G$ acts  $\Gamma$-equivariantly on $\hat M$, then it acts  $\Gamma$-equivariantly on $\hat M\times X$, trivially in the second component.  This induces a $G$-action on $X_\rho$.  This suspension construction is also used to construct the Kontsevich-Zorich cocycle in the next section.

Returning to lattices in $SL(2,\reals)$, the following result gives a concrete criterion for establishing when the Lyapunov exponents of the $A$-action on $\bH_\rho$ with respect to Haar measure on $SL(2,\reals)/\Gamma$ are all distinct, a property often referred to as {\em simple Lyapunov spectrum}.

\begin{theorem}\label{t=EskinMatheus}\cite[Corollary 5.5]{Bader:Furman:ICM} Let $\Gamma < SL(2,\reals)$ be a lattice and $\rho\colon \Gamma\to {SL(H)}$ be a representation.
If  $\rho(\Gamma)$ is Zariski dense in {$SL(H)$,} then the Lyapunov spectrum for the $A$-action on $\bH_\rho$, with respect to Haar measure on $G/\Gamma$, is simple.
\end{theorem}

Theorem~\ref{t=EskinMatheus} is proved in the context of
Kontsevich-Zorich cocycles in \cite[Theorem
1]{Eskin:Matheus:Coding:Free}, {by combining the
  main result of \cite{Goldsheid:Margulis} and \cite{GR}
  with \cite[Theorem~3]{Furstenberg:bullitin}, (whose proof is given
  in \cite{Furstenberg:randomwalks}).} It is straightforward to adapt
the proof in \cite{Eskin:Matheus:Coding:Free} to give a proof of Theorem~\ref{t=EskinMatheus}.  The formulation in \cite{Bader:Furman:ICM} is more general than stated here and fits into a general theory (building on a long tradition initiated by Furstenberg) connecting Lyapunov spectrum with the theory of boundary actions.

\subsubsection{Foliated geodesic and horocyclic flows}

The suspension construction gives a natural way to construct a foliated bundle with hyperbolic leaves and compact leaf space.  Namely, given a representation $\rho\colon \Gamma \to GL(H)$ of a lattice $\Gamma<SL(2,\reals)$,  one constructs the  projective bundle  ${\mathbb P}(\check\bH_\rho)$ over  the hyperbolic surface $\bbH/\Gamma$
by suspending the induced  projective action of $\Gamma$ on $\proj(H)$, with $\Gamma$ acting by isometries on the first factor of $\check\bH = \bbH\times \proj(H)$. 

The  bundle  ${\mathbb P}(\check\bH_\rho)$ does not carry a $SL(2,\reals)$ action, but it does carry a natural {\em  foliation} (equivalently, a flat connection).  In particular, the (trivial) horizontal foliation of $\check\bH$ by hyperbolic planes descends under the $\Gamma$-quotient to a foliation  $\cF_\rho$ of  ${\mathbb P}(\check\bH_{\rho})$ by hyperbolic surfaces.  The leaf space of this foliation is $\proj(H)$, and the monodromy representation for $\cF_\rho$ is $\rho$.     Since the leaves of $\cF_\rho$ carry a Riemannian structure, the unit tangent bundle $T^1\cF_\rho$ to the foliation $\cF_\rho$ carries a natural flow, the foliated geodesic flow.  The construction of the foliated geodesic flow of a representation can be carried out quite generally, whenever $\Gamma = \pi_1(M)$, with $M$ a Riemannian manifold  (see, e.g. \cite[Example 5.2]{Bader:Furman:ICM}). 

Now recalling that the unit tangent bundle to $\bbH/\Gamma$ is $SL(2,\reals)/\Gamma$, we note that in the present setting
$T^1\cF_\rho$ is naturally identified with the bundle $\proj(\bH_\rho)$, obtained by again suspending the action of $\Gamma$ on  $\proj(H)$, this time with $\Gamma$ acting by right multiplication on the first factor of  $\bH = SL(2,\reals)\times \proj(H)$.    There is a natural projection  $\proj(\bH_\rho)$ to $\proj(\check\bH_\rho)$ induced by the
projection $SL(2,\reals)/\Gamma \to \bbH/\Gamma$.  The bundle  $\proj(\bH_\rho)$  admits an $SL(2,\reals)$ action;
the action of $A$ on  ${\mathbb P}(\bH_\rho)$  is the foliated geodesic flow, and the action of $U_+$ is the foliated (positive) horocycle flow.
 
Note that the bundle $\proj(\bH_\rho)$ is the projectivization of the bundle $\bH_\rho\to SL(2,\reals)/\Gamma$, which is the suspension of the $\rho$-action on $H$ defined above. This bundle also admits a $SL(2,\reals)$ action.
Bonatti and Gomez-Mont proved the following result, which connects the dynamics of the foliated horocycle flow $U_+$ on $\proj(\bH_\rho)$  with the Lyapunov exponents of the $A$-action on $\bH_\rho$.

\begin{theorem}[\cite{Bonati:Comez-Mont}] Let  $\rho\colon \Gamma\to GL(3, {\mathbb C})$ satisfy the integrability condition (\ref{eq:cocycle:integrability}), which holds automatically if $G/\Gamma$ is compact.    If there is no $\rho(\Gamma)$-invariant measure on  ${\mathbb C}{\mathbb P}^2$, and the largest Lyapunov exponent for the $A$-action on $\bH_\rho$  has multiplicity $1$, then there is a unique probability measure invariant under the foliated horocycle flow on ${\mathbb P}(\bH_\rho)$ and projecting to Lebesgue/Haar measure on $SL(2,\reals)/\Gamma$.  In particular, if $SL(2,\reals)/\Gamma$ is compact, then the  foliated horocycle flow   is uniquely ergodic.
\end{theorem}

{Our result generalizes the Bonatti--Gomez-Mont theorem to linear representations of $\Gamma$ in arbitrary dimension: }
\begin{theorem}\label{t=uniqueergodicfoliation}  Let $H={\mathbb C}^k$ or $\reals^k$, for $k\geq 2$, and let $\rho\colon \Gamma\to GL(H)$ satisfy the integrability condition (\ref{eq:cocycle:integrability}), which holds automatically if $G/\Gamma$ is compact.   If the representation $\rho$ is irreducible (i.e.,  if  $\rho(\Gamma)$ has no proper invariant subspace of $H$), and the largest Lyapunov exponent for the $A$-action on $\bH_\rho$ has multiplicity $1$, then  there is a unique probability measure invariant under the horocycle flow on ${\mathbb P}(\bH)$ and projecting to Lebesgue/Haar measure on $SL(2,\reals)/\Gamma$.  In particular, if $SL(2,\reals)/\Gamma$ is compact, then the foliated horocycle flow   is uniquely ergodic.
\end{theorem}

\begin{proof} We first observe that the condition that $\rho$ is an irreducible representation implies that the $G$-action on $\bH_\rho$ is irreducible.   Suppose
that there were a proper, nontrivial $G$-invariant $\nu$-measurable subbundle of $\bH$.  Since the base $G/\Gamma = T^1\Sigma$ consists of a unique $G$-orbit,  this bundle must be continuous (indeed analytic). This bundle projects to a continuous subbundle in $\bH_{\Sigma,\rho}$ invariant under the monodromy of $\cF_\rho$.   Fixing a point $x\in X$ and considering the monodromy representation of $\pi_1(x,\Sigma) $ on the fiber of  $\bH_{\Sigma,\rho}$ over $x$, this monodromy-invariant subbundle of $\bH_{\Sigma,\rho}$ restricts to a  $\rho(\Gamma)$-invariant subspace of $H$, which gives a contradiction.


Since $\bE^1$ is one-dimensional, the fibers
of $\proj(\bE^1)$ consist almost everyhere of a single point, and so $\proj(\bE^1)$ is the image of a measurable section
$\sigma\colon SL(2,\reals)/\Gamma \to \proj(\bH)$.
Let $\mu^+ := \sigma_\ast\nu$  be the pushforward of Lebesgue/Haar measure $\nu$ on $SL(2,\reals)/\Gamma$ to $\proj(\bH)$ under $\sigma$, which is invariant under the $P$ action and ergodic under both $A$ and $U_+$ actions.

Suppose that $\mu$ is another measure on invariant under $U_+$ and projecting to $\nu$.
If   $\mu$ is absolutely continuous with respect to $\mu^+$, then  ergodicity of $\mu^+$ with respect to $U_+$ implies that $\mu=\mu_+$.  
We may assume that $\mu$ is not absolutely continuous with respect to $\mu_+$  
Taking the singular part of $\mu$, saturating by the $U_+$ action,  and renormalizing to be a probability, we obtain a measure  $\mu_1$ that is $U_+$ invariant and singular with respect to  $\mu^+$.     The projection of this measure to $G/\Gamma$  is absolutely continuous with respect to 
$\nu$ and hence is equal to $\nu$. The disintegration of $\mu_1$ along fibers assigns measure $0$ to $\proj(\bE^1)$.

Now push forward $\mu_1$ along the  action of  $a^{-t}$ and average, obtaining a limit measure $\nu$ that is $P$-invariant and projects to $\nu$.  
The details of this construction are worked out in \cite[Sections 4 and 5]{Bonati:Comez-Mont}.

 Since we took averages over negative time, Oseledets' theorem implies that $\nu$ is not supported on $\proj(\bE^1)$.  But this is a contradiction, by Theorem~\ref{theorem:P:uniq:ergodicity}, since the $G$-action on $H$ is irreducible.
\end{proof}

In light of Theorem~\ref{t=EskinMatheus}, we obtain a simple criterion for unique ergodicity of the foliated horocycle flow when $\Gamma$ is cocompact.


\begin{corollary}  \label{c=uniqueergodicfoliation}  Let $H={\mathbb C}^k$ or $\reals^k$, for $k\geq 2$, and let $\rho\colon \Gamma\to GL(H)$ satisfy the integrability condition (\ref{eq:cocycle:integrability}).   If the image  $\rho(\Gamma)$ is Zariski dense in $GL(H)$, then  there is a unique probability measure invariant under the horocycle flow on ${\mathbb P}(\bH)$ and projecting to Lebesgue/Haar measure on $SL(2,\reals)/\Gamma$.  

In particular, if $SL(2,\reals)/\Gamma$ is compact and $\rho\colon \Gamma\to GL(H)$ is any representation with Zariski dense image,  then the foliated horocycle flow is uniquely ergodic.
\end{corollary}

In light of the results in \cite{BakhtinMartinez}, we also obtain the following immediate corollary to Theorem~\ref{theorem:P:uniq:ergodicity}, which we state for simplicity in the cocompact setting.

\begin{corollary} Let $\Gamma<SL(2,\reals)$ be a cocompact lattice,  let $H={\mathbb C}^k$ or $\reals^k$, for $k\geq 2$, and let $\rho\colon \Gamma\to GL(H)$ be a representation such that $\rho(\Gamma)$ is Zariski dense.

 Consider the compact manifold $M={\mathbb P}(\check\bH_{\rho})\to \bbH/\Gamma$, and let $\cF_\rho$ be the horizontal foliation by hyperbolic surfaces.
 Then there is a unique probability measure on $M$ that is harmonic with respect to $\cF_\rho$.

\end{corollary}

Matsumoto \cite{Matsumoto} recently showed that the same results do not hold for general foliations by hyperbolic surfaces:  on a ($3$-dimensional) solvmanifold, he constructs a foliation with hyperbolic leaves whose foliated horocycle flow is not uniquely ergodic and admitting more than one harmonic measure.

\subsection{The Kontsevich-Zorich cocycle}\label{ss=KZ}
$  $

{ 
In our second application of Theorem~\ref{theorem:P:uniq:ergodicity}, we deduce the continuity of Lyapunov exponents of affine invariant measures.

\bold{Flat surfaces and strata.}  Suppose $g \ge 1$, and let $\beta =
(\beta_1,\dots, \beta_m)$ be a partition of $2g-2$.  
{Let the
Teichm\"uller space $\Omega(\beta)$ be the space of pairs $(M,\omega)$
where $M$ is a Riemann surface of genus $g$ and $\omega$ is a
holomorphic 1-form on $M$ (i.e.\ an Abelian differential) whose zeroes
are labelled from $1$ to $m$, and such that the zero labeled $i$ has
multiplicity $\beta_i$. In $\Omega(\beta)$, $(M,\omega)$ and
$(M',\omega')$ are identified if there exists a holomorphic
diffeomorphism $f: M \to M'$ isotopic to the identity, pulling back
$\omega'$ to $\omega$ and preserving the labeling of the zeros. The
space $\Omega(\beta)$ is smooth and contractible.}

{
We also consider the moduli space $\cH(\beta)$ whose definition is the
same as that of $\Omega(\beta)$ except that in $\cH(\beta)$,
$(M,\omega)$ and $(M',\omega')$ are identified if there exists a
holomorphic diffeomorphism $f: M \to M'$ (not necessarily isotopic to
the identity) pulling back $\omega'$ to $\omega$ and preserving the
labeling of the zeroes. We will refer to $\cH(\beta)$ as a {\em
stratum of Abelian differentials}.
}

{
Let $S$ be a closed topological surface of genus $g$, and let $\Sigma$
be a collection of $m$ pairwise distinct points in $S$.  The group
$\hbox{Diff}^+(S,\Sigma)$ of diffeomorphisms of $S$ fixing the points
in $\Sigma$ acts on $ \Omega(\beta)$, with stabilizer containing the
subgroup $\hbox{Diff}_0(S,\Sigma)$ of null-isotopic
diffeomorphisms. Considering the quotient action, one obtains a
properly discontinuous action of the mapping class group
$\hbox{Mod}(S,\Sigma) :=
\hbox{Diff}^+(S,\Sigma)/\hbox{Diff}_0(S,\Sigma)$ on $ \Omega(\beta)$,
with quotient $\cH(\beta)$. Thus, the space $ \cH(\beta)$ is an
orbifold with universal cover $ \Omega(\beta)$.  For any $\beta$,
there is a $k$-fold finite cover $ \cH(\beta)_k$ of $\cH(\beta)$ in
which these orbifold points disappear: $ \cH(\beta)_k$ is a smooth
manifold covered by $\Omega(\beta)$.
}

For each $(M,\omega)\in  \cH(\beta)$, 
the form $\omega$ defines a
canonical flat metric on $M$ with conical singularities at the zeros
of $\omega$. Thus we refer to points of $ \cH(\beta)$ as
{\em flat surfaces} or {\em translation surfaces}. For an introduction
to this subject, see the survey \cite{Zorich:survey}. 

\bold{Affine measures and manifolds.} 
There is a well-studied action of  $G=SL(2,\reals)$ on $ \Omega(\beta)$; in this action, $g\in G$ acts linearly on translation structures via post-composition in charts.  This action preserves the subset
$ \Omega_1(\beta) \subset  \Omega(\beta)$ of marked surfaces of
(flat) area $1$.   This $G$-action commutes with the action of $\hbox{Mod}(S,\Sigma)$ and so descends to an action 
on  $ \cH(\beta)$ preserving the space  $ \cH_1(\beta)$ of flat surfaces of area $1$.}

An {\em affine invariant manifold} is
a closed subset of $ \cH_1(\beta)$ that is invariant under this
$G$ action and which in \textit{period coordinates} (see
\cite[Chapter ~3]{Zorich:survey}) looks like an affine subspace. Each affine invariant
manifold $\cN$ is the support of an ergodic $SL(2,\reals)$-invariant
probability measure $\nu_\cN$. Locally, in period coordinates, this
measure is (up 
to normalization) the restriction of Lebesgue measure to the subspace
$\cN$,  see \cite{EM} for the precise definitions. 
It is proved in \cite{EMM} that the closure of any
$SL(2,\reals)$ orbit is an affine invariant manifold; this is
analogous to one of Ratner's theorems in the theory of unipotent
flows. (In genus $2$, this result was previously proved by McMullen
\cite{McMullen:genus2}). 

To state our result, we will need the following:
\begin{theorem}[\protect{\cite[Theorem 2.3]{EMM}}]
\label{theorem:mozes-shah}
Let $\cN_n$ be a sequence of affine manifolds in $ \cH(\beta)$, and suppose
$\nu_{\cN_n} \to \nu$.  Then $\nu$ is a probability measure.
Furthermore, $\nu$ is the affine measure $\nu_\cN$, where $\cN$ is the
smallest submanifold with the following property: there exists some
$n_0 \in \natls$ such that $\cN_{n}\subset\cN$ for all $n>n_0$.

In particular, the space of ergodic $P$-invariant probability measures on
  $ \cH_1(\beta)$ is compact in the weak-star topology. 
\end{theorem}

{  Since  $ \cH_1(\beta)$ is not compact, the weak star topology is
  defined using compactly supported functions.  In particular,
  $\nu_{\cN_n}\to \nu$ if and only if for any compactly supported
  continuous function $\phi$ on $\cH_1(\beta)$, we have that $\int \phi\,d\nu_{\cN_n} \to \int \phi\,d\nu$.}

\begin{remark}
\label{remark:mozes:shah}
In the setting of unipotent flows, Theorem~\ref{theorem:mozes-shah} is
due to Mozes and Shah \cite{MS}.
\end{remark}

\bold{The Kontsevich-Zorich cocycle.}{   The \textit{Hodge bundle} $\bH$ over $ \cH(\beta)$
is constructed in a similar fashion to the suspension construction in Subsection~\ref{ss=monodromy}.
We start with the bundle $\hat{\bH}$ over $ \Omega(\beta)$ whose fiber over  $(M,\omega)$ is the cohomology group
$H^1(M,\reals)$ (viewed as a $2g$-dimensional real vector space).   As $ \Omega(\beta)$ is contractible, this
bundle is trivial, diffeomorphic to $ \Omega(\beta)\times \reals^{2g}$.  We then consider the natural (right) action of the group $\Gamma = \hbox{Mod}(M,\Sigma)$ 
on  $\hat{\bH}$, where $\gamma\in \Gamma$ acts by pullback in the fibers, and we set
$\bH : = \hat{\bH}/\Gamma$, which is a bundle over $ \cH(\beta)$.
The Hodge bundle fails to be 
a vector bundle over the orbifold points of $ \cH(\beta)$, but passing to a finite branched cover, the corresponding 
Hodge bundle over   $ \cH(\beta)_k$ is a smooth vector bundle, and the $G$-action lifts to this bundle.  As the results below concern affine measures, which behave well under finite branched covers, {\em we will henceforth assume that $\bH$ is a vector bundle over $ \cH(\beta)$.}

There is a natural (left) $G$-action on $\hat{\bH}$ that is the standard $G$ action on $ \Omega(\beta)$ and trivial on the fibers.  This descends as in Subsection~\ref{ss=monodromy} to a $G$ action on $\bH$.
This $G$~action is often described by measurably trivializing $\bH$ and describing elements of this
action using elements of the symplectic group $Sp(2g,\zed)$.  To be precise, if we
choose a fundamental domain  in $ \Omega(\beta)$ for the action of the mapping class group
$\Gamma$, then we have the cocycle $\tilde{\alpha}: SL(2,\reals) \cross
 \cH_1(\beta) \to \Gamma$ where for $x$ in the fundamental domain,
$\tilde{\alpha}(g,x)$ is the element of $\Gamma$ needed to return the point
$gx$ to the fundamental domain. Then, we define the Kontsevich-Zorich
cocycle $\alpha(g,x)$ by 
\begin{displaymath}
\alpha(g,x) = \rho(\tilde{\alpha}(g,x)),
\end{displaymath}
where $\rho: \Gamma \to Sp(2g,\zed)$ is the homomorphism given by the
action on cohomology. The Kontsevich-Zorich cocycle can be interpreted as the
monodromy of the Gauss-Manin connection restricted to the orbit of
$SL(2,\reals)$, see e.g. \cite[page~64]{Zorich:survey}. 

Henceforth we will conflate action and cocycle and simply refer to this  $G$-action on $\bH$ as the Kontsevich-Zorich cocycle.   We further restrict (i.e. pull back) the bundle $\bH$ to  the space  $ \cH_1(\beta)$ of flat surfaces of area $1$.  As the $G$-action on $ \cH(\beta)$ preserves  $ \cH_1(\beta)$, we thus have a restricted $G$ action on the bundle
$\bH\to  \cH_1(\beta)$.  The discussion that follows concerns this action (cocycle) on this bundle.

There is a way to choose a Finsler structure on the Hodge bundle $\bH$
over $ \cH_1(\beta)$ so that the Kontsevich-Zorich cocycle satisfies
the integrability condition (\ref{eq:cocycle:integrability}) with
respect to {\em any} affine measure $\nu$ on any such $\cF$.  Moreover
this integrability is uniform: {there exists $C$ depending only on the
genus such that for any $x \in \cH_1(\beta)$, any $t \in [-1,1]$
and any $\bfv \in \bH(x)$,
\begin{displaymath}
C^{-1} \| \bfv \|_x \le \| a^t_* \bfv \|_{a^t x} \le C \| \bfv\|_x.
\end{displaymath}
(see \cite[Subsection 2.2.2]{Avila:Gouezel:Yoccoz} or \cite[Section
7.2]{EMM}). In particular,}
\begin{equation}
\label{eq:cocycle:uniformintegrability}
\int_Z \sup_{t\in [-1,1]} \left(\log\|a^t_* \|_x \right) \, d\nu(x) < \epsilon.
\end{equation}
The Lyapunov exponents of the Kontsevich-Zorich cocycle are defined with respect to this Finsler structure;
they exist and are constant almost everywhere with respect to any ergodic affine measure.}
These exponents appear in
other areas of dynamics, for example in the study of deviations of ergodic
averages of interval exchange tranformations \cite{Zorich:Deviation}
\cite{Forni:Deviation}, \cite{Bufetov} and in the study of the
diffusion rate of the periodic wind-tree model \cite{DHL:Windtree:diffusion}.

The Lyapunov spectrum of the Kontsevich-Zorich cocycle has been
studied extensively. Building on numerical experiments by Zorich, it was conjectured by Kontsevich in
\cite{Kontsevich} 
that for the Masur-Veech measures, the Lyapunov spectrum is simple: this Kontsevich-Zorich Conjecture
has been proved partially by Forni in \cite{Forni:Deviation} 
and then fully by Avila and Viana \cite{Avila:Viana}. For other
affine measures $\nu$ the situation is more complicated (and in
particular the spectrum need not be simple). Also the sum of the
Lyapunov exponents (and in some cases the Lyapunov exponents
themselves) can be computed explicitly. Recent papers in which the
Lyapunov spectrum of the Kontsevich-Zorich cocycle plays a major role include
\cite{Aulicino:thesis,
Aulicino:affine,
Aulicino:Lyapunov:genus3,
Athreya:Eskin:Zorich,
Bainbridge:Euler:char,   
Bainbridge:L:shaped, 
Bouw:Moeller,
Chen:Moeller:abelian,
Chen:Moeller:quadratic,
Delecroix:Matheus,
Delecroix:Zorich,
Deroin:Daniel,
EKMZ:lower:bounds,
Eskin:Kontsevich:Zorich,                
Eskin:Kontsevich:Zorich2, 
Eskin:Matheus:Coding:Free,
Filip,
Filip2,
Filip3,
Filip:Forni:Matheus,
Forni:Handbook,
Forni:new, 
Forni:Matheus:example,
Forni:Matheus:Zorich,
Forni:Matheus:Zorich:second:fund:form,
Forni:Matheus:Zorich:semisimplicity,
Kontsevich:Zorich, 
Hubert:Bourbaki,
Grivaux:Hubert:in:progress,
Matheus:appendix, 
Matheus:Yoccoz, 
Matheus:Yoccoz:Zmiaikou,
Matheus:Moeller:Yoccoz:Criterion,
Matheus:Weitze-Schmithusen,
Moeller:Shimura:curves,
Pardo:counting,
Pardo:error:term,
Pardo:non:varying,
Trevino,
Fei:Yu:conjecture,
Yu:Zuo:Weierstrass,
Wright:Abelian}.

The following theorem answers a question asked in 
\cite{Matheus:Moeller:Yoccoz:Criterion}.
\begin{theorem}
\label{theorem:convergence:of:Lyapunov:exponents}
Let $\cN_n$ be a sequence of affine manifolds, and suppose
the affine measures $\nu_{\cN_n}$ converge to the (affine) measure
$\nu$ (as in Theorem~\ref{theorem:mozes-shah}). 
Then the Lyapunov exponents of  $\nu_{\cN_m}$ (with respect to $a^t$) converge to the
Lyapunov exponents of $\nu$. 
\end{theorem}

The proof of Theorem~\ref{theorem:convergence:of:Lyapunov:exponents}
depends on Theorem~\ref{theorem:P:uniq:ergodicity} and the following
theorem of S.~Filip:

\begin{theorem}[\protect{\cite{Filip}}]
\label{theorem:filip}
Let $\alpha$ denote (some exterior power of) 
the Kontsevich-Zorich cocycle restricted
to an affine invariant submanifold $\cN$. Let $\nu_\cN$ be the affine
measure whose support is $\cN$, and suppose $\alpha$ has a
$\nu_\cN$-measurable almost-invariant subbundle $\bW$ (see  Definition~\ref{def:almost:invariant:splitting}). Then $\bW$ agrees $\nu_\cN$-almost everywhere with a continuous one.
\end{theorem}

In fact it is proved in \cite{Filip} that the
  dependence of any such $\bW(x)$ on $x$ is polynomial in the period coordinates.

\bold{Simplicity of Lyapunov Spectrum of Teichm\"uller curves.}
We recall that
a \emph{Teichm\"uller curve} is a closed $SL(2,\reals)$
orbit on $ \cH_1(\beta)$. Teichm\"uller curves (which are submanifolds
of real dimension $3$) are the smallest
possible affine manifolds; any other type of affine manifold has
dimension greater than $3$.

By the Kontsevich-Zorich classification
\cite{Kontsevich:Zorich:connected:components}, 
the stratum $ \cH(4)$ (in genus $3$) has two connected connected components
$ \cH(4)^{hyp}$ and $ \cH(4)^{odd}$. The connected component
$ \cH(4)^{hyp}$ consists entirely of hyperelliptic surfaces.

As a consequence of some recent results, we obtain the following:
\begin{theorem}
\label{theorem:Simplicity:H4hyp}
All but finitely many Teichm\"uller curves in 
$ \cH(4)$ have simple Lyapunov spectrum. 
\end{theorem}

Theorem~\ref{theorem:Simplicity:H4hyp} was shown in 
\cite{Matheus:Moeller:Yoccoz:Criterion} to follow from a conjecture
of Delecroix and Leli\`evre (stated in
\cite{Matheus:Moeller:Yoccoz:Criterion}). 
Our proof below is unconditional; however it is much less explicit
and is completely ineffective. It also 
depends on the very recent results of Filip \cite{Filip},
Nguyen-Wright \cite{Nguyen:Wright}, and \cite{Aulicino:Nguyen:Wright}.

\bold{Proof of Theorem~\ref{theorem:Simplicity:H4hyp}.} 
Suppose there exist infinitely many Teichm\"uller curves $\cN_n
\subset  \cH_1(4)$ with multiplicities in the Lyapunov
spectrum. By Theorem~\ref{theorem:mozes-shah}, the $\cN_n$ have to
converge to an affine manifold $\cN$ (in the sense that the affine
measures $\nu_{\cN_n}$ will converge to the affine measure $\nu_\cN$).
Furthermore, $\cN$ cannot be a Teichm\"uller curve (since by
Theorem~\ref{theorem:mozes-shah}, $\cN$ must contain all the $\cN_n$
for $n$ sufficiently large, and thus $\dim \cN > 3$). 
By the main theorems of
\cite{Nguyen:Wright} and \cite{Aulicino:Nguyen:Wright}  
the only proper affine submanifolds of
$ \cH_1(4)$ that are not Teichm\"uller curves  are the two connected
components $ \cH_1(4)^{hyp}$ and $ \cH_1(4)^{odd}$ and the Prym locus
$\cP$ of $\cH_1(4)^{odd}$. 
 Thus $\cN$ must be either 
$ \cH_1(4)^{hyp}$ or $ \cH_1(4)^{odd}$ or $\cP$. 
The Lyapunov spectrum of both $ \cH_1(4)^{hyp}$ and $ \cH_1(4)^{odd}$
is simple by \cite{Avila:Viana}, and the Lyapunov spectrum of $\cP$ is
simple by \cite[\S{6.6}]{Matheus:Moeller:Yoccoz:Criterion}
This contradicts
Theorem~\ref{theorem:convergence:of:Lyapunov:exponents}. 
\qed\medskip

\begin{remark}
\label{remark:improve:applicability}
We expect that the applicability of the proof of
Theorem~\ref{theorem:Simplicity:H4hyp} to increase as our knowledge of
the classification of affine invariant manifolds increases. 
\end{remark}

\begin{remark} It remains a challenge to find an approach to the
Delecroix-Leli\`evre conjecture itself which would leverage
  some equidistribution results. 
\end{remark}

\section{Construction Of An Invariant Subspace}
\label{sec:first:divergence}

We write $\cB := \{u_+^s \st s\in [-1,1] \}$ for the ``unit ball''  in $U_+$, and $\cB_t$ for $a^t
\cB a^{-t} = \{ u_+^s \st s\in [-e^{2t} ,e^{2t}]  \}$. 
We also write
\begin{displaymath}
\cB[x] = \{ u x \st u \in \cB \} \subset X, \qquad \cB_t[x] = \{ u x \st u \in \cB_t \} \subset X.
\end{displaymath}
Similarly, we use the notation
\begin{displaymath}
U_+[x] = \{ u x \st u \in U_+ \} \subset X.
\end{displaymath}

Recall that $\nu$ denotes a $G$-invariant measure on $X$ that is ergodic for the action of $A=\{a^t\}$. 
In what follows, when we write
``almost all $x$'' we mean ``all $x$ outside a set of
$\nu$-measure $0$."

\subsection{The forward and backward flags}
We consider the forward and backward flags for the action of $a^t$ on fibers, which define the Oseledets decomposition; let
\begin{displaymath}
\{0\} =  \bE_{\geq n+1} (x) \subset \bE_{\geq n}(x) \subset \dots \subset
\bE_{\geq 1}(x) = \bH(x)
\end{displaymath}
be the forward flag, and let
\begin{displaymath}
\{0\} =  \bE_{\leq 0}(x) \subset \bE_{\leq 1}(x) \subset \dots \subset \bE_{\leq n}(x) = \bH(x)
\end{displaymath}
be the backward flag. 
This means that for almost all $x$ and 
for $\bfv \in \bE_{\geq i}(x)$ such that $\bfv \not\in \bE_{\geq i+1}(x)$, 
\begin{equation}
\label{eq:bVi:growth:exactly:lambda:i}
\lim_{t \to + \infty} \frac{1}{t} \log \frac{\| a^t_\ast \bfv\|}{\|\bfv\|}
= \lambda_i,
\end{equation}
and for $\bfv \in \bE_{\leq i}(x)$ such that $\bfv \not\in  \bE_{\leq i-1}(x)$, 
\begin{equation}
\label{eq:growth:hat:V:n:plus:one:minus:i}
\lim_{t \to - \infty} \frac{1}{t} \log \frac{\|a^t_\ast \bfv\|}{\|\bfv\|} =
\lambda_{i}.
\end{equation}

By e.g. \cite[Lemma 1.5]{Goldsheid:Margulis},  for a.e. $x$ and every $i$, we have 
\begin{equation}
\label{eq:Lyapunov:transversality}
\bH(x) = \bE_{\leq i}(x) \dirsum \bE_{\geq i+1}(x).
\end{equation}
The bundles $\bE_1, \ldots ,\bE_n$ in the Oseledets decomposition  are then defined by the almost everywhere transverse intersections 
$\bE_i := \bE_{\leq i}\cap \bE_{\geq i}$.
Note that  $\bE_{\leq 1} = \bE_1$, $\bE_{\geq n} = \bE_n$ and in general,
$\bE_{\leq j} = \bE_1 \oplus \bE_2 \oplus \cdots \bE_j$ and 
$\bE_{\geq j} = \bE_j \oplus \bE_{j+1} \oplus \cdots \bE_n$ almost everywhere.

The bundles $\bE_{\geq i}$ and  $\bE_{\leq i}$ are evidently $A$-equivariant; for almost all $x \in X$ and every $t\in \bbR$, we have
\begin{equation}
\label{eq:invariance:g}
a^t_*\left( \bE_{\leq j}(x) \right)=  \bE_{\leq j}(a^t x), \qquad  a^t_*\left( \bE_{\geq j}(x)\right) =
\bE_{\geq j}(a^t x). 
\end{equation}

\begin{lemma}
For almost all $x \in X$, $u \in U_+$ and $\bar{u} \in U_-$, we have
\begin{equation}
\label{eq:invariance:Vj:hatVj}
u_* \left(\bE_{\leq j}(x) \right)=  \bE_{\leq j}(ux), \qquad  \bar{u}_* \left(\bE_{\geq j}(x) \right)=
\bE_{\geq j}(\bar{u} x). 
\end{equation}
\end{lemma}
\begin{proof}Equation (\ref{eq:bVi:growth:exactly:lambda:i}) means that the bundle $\bE_{\geq j}(x)$ is characterized by the fact that the forward rate of expansion of vectors in $\bE_{\geq j}(x)$ under $a^t_\ast$
is no more than $\lambda_j$. For any $s,t\in\bbR$ the vectors $\bfv\in H_x$, for some $x\in X$, and $\bfw=(u^s_-)_*(\bfv)$ satisfy 
$$a^t_*(\bfw)=(u^{exp(-t)\cdot s}_-)_*(a^t_*(\bfv)).$$
 
Since the measure $\nu$ is $G$-invariant, almost every point $x$ is recurrent for $a^t$. 
Thus  the sequence of linear maps $(u^{exp(-t)\cdot s}_-)_*$ tend to the identity for large $t>0$ for which $a^t_*x$ is close to $x$. 
As a consequence, if the rates of expansion of the vectors $\bfv$ and $\bfw$ are defined, they are the same. 
This proves the invariance of the bundle  $\bE_{\geq j}(x)$  under $U_-$.
\end{proof}

\begin{remark} If for some $j\geq 2$, the bundle $\bE_{\geq j}$  were also invariant under almost every $u\in U_+$, then because $G$ is generated by $A, U_+$ and $U_-$, 
the bundle would be invariant under $G$, violating the irreducibility hypothesis on the action of $G$.
\end{remark}

As remarked before, the bundles $\bE_i$ and therefore $\bE_{\geq i}$ and $\bE_{\leq i}$ are measurable. This implies that, for any Borel probability measure for which these
bundles are almost everywhere defined, they are continuous when restriction to a set  whose measure is arbritrarily close to $1$. This implies the following statement.

\begin{lemma} Let $\bE\subset \bH$ be a measurable subbundle of $\bH$. For any compact subset $K\subset X$ and any probability measure $\theta$ on $K$ with respect to which
the bundle is defined $\theta$-almost everywhere, there is a set $Y\subset X$ of full $\theta$-measure such that any $y\in Y$ admits a subset $G_y\subset Y$ with the following properties: 
\begin{itemize}
 \item $\theta(G_y\cap U)>0$, for any neighborhood $U$ of $y$;
 \item the bundle $\bE$ is continuous on $G_y$. 
\end{itemize}
\end{lemma}

The map from $\reals$ to $U_+$ defined by $t\mapsto u^+_t$ defines a
measure on $U_+$ by the pushforward of Lebesgue measure on
$\reals$. {This measure coincides with the Haar measure
  on $U_+$.} We refer to this measure as Lebesgue measure on $U_+$ and
denote it by $|\cdot |$.   The previous lemma, the $U_+$-invariance of $\nu$ and Fubini's theorem imply:

\begin{corollary}\label{c.continuity} In our setting there is a subset $Y\subset X$ of full $\nu$-measure such that for all $y \in Y$:
\begin{itemize}
 \item the bundles $\bE_{\geq j}(uy)$ are  well-defined for Lebesgue-almost every $u \in \cB$;
 \item there exists a  measurable set $\cG_y \subset \cB$ such that $|\cG_y \cap U|>0 $, for any neighborhood $U$ of the identity in $\cB$, and such that  $u\mapsto \bE_{\geq j}(uy)$ is continuous on $\cG_y$.
\end{itemize}

\end{corollary}

\subsection{The forward flag and the unstable horocycle flow: defining the inert flag.}

Our strategy for proving Theorem~\ref{theorem:P:uniq:ergodicity} is to show that if the measure $\hat \nu$ is not supported in $\bE_1$, 
then there is a nontrivial subbundle of some $\bE_{\geq j}$ that is  $G$-invariant, obtaining a contradiction.


{
To this end, we define the {\em inert flag}
\begin{equation}
\label{eq:flag:bFj}
\{0\} = \bF_{\geq n+1}(x) \subseteq \bF_{\geq n}(x) \subseteq \bF_{\geq n-1}(x) \subseteq
\dots \bF_{\geq 2}(x) \subseteq \bF_{\geq 1}(x) = \bH(x).
\end{equation}
where $\bF_{\geq j}(x)$, for $j=1, \ldots, m$ are defined by
\begin{equation}
\label{eq:def:bFj}
\bF_{\geq j}(x) = \{ \bfv \in \bH(x) \st \text{ for almost 
all $u \in \cB$, $u_* \bfv \in \bE_{\geq j}(u x)$} \}.
\end{equation}
In other words, $\bfv \in \bF_{\geq j}(x)$ if and only if,
for almost all $u \in \cB$ we have
\begin{equation}
\label{eq:meaning:bFj}
\limsup_{t \to \infty} \frac{1}{t} \log \|a^t_* u_* \bfv\| \le \lambda_j.
\end{equation}
From the definition it follows that the $\bF_{\geq j}(x)$ are  vector subspaces of $\bH(x)$ that
decrease with $j$, and thus they indeed form a flag. 
}


\begin{lemma}
\label{lemma:Ej:equivariant}
For almost every $x\in X$, the following hold:
\begin{itemize}
\item[{\rm (a)}] $\bF_{\geq j}(x) \subset  \bE_{\geq j}(x)$;
\item[{\rm (b)}] $\bF_{\geq j}(x) = \{ \bfv \in \bH(x) \st \text{ for almost 
all $u \in U_+$, $u_* \bfv \in \bE_{\geq j}(u x)$} \}$;

\item[{\rm (c)}]  $\bF_{\geq j}(x)$ is $A$-equivariant; i.e.,  
$a^t_*\left( \bF_{\geq j}(x) \right)= \bF_{\geq j} (a^t x)$, for all $t\in\bbR$; and

\item[{\rm (d)}] for almost all $u \in U_+$,  we have 
$u_* \left(\bF_{\geq j}(x)\right) = \bF_{\geq j}(ux)$.


\end{itemize}
\end{lemma}

\begin{figure}[h]
\includegraphics[scale=0.35]{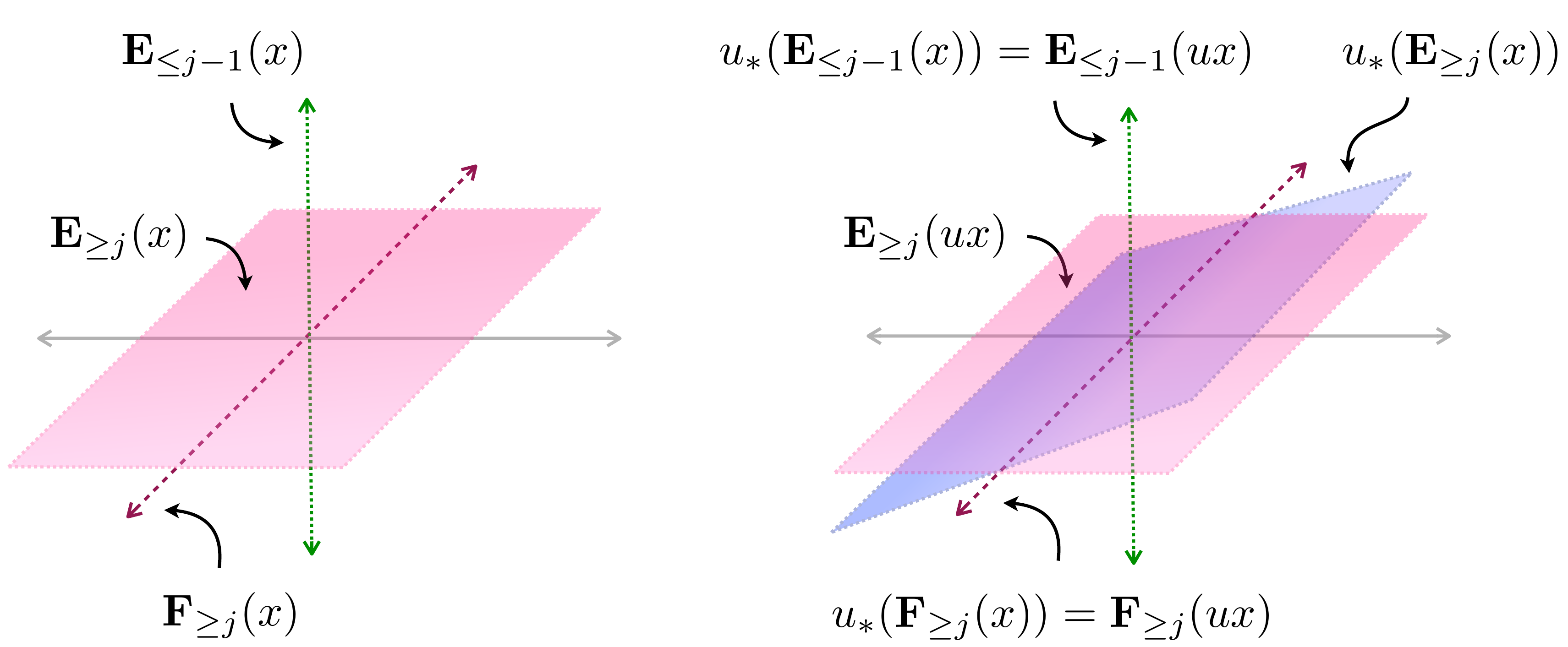}
\caption{Invariance properties of $\bE_{\leq j-1}, \bE_{\geq j}$, and $\bF_{\geq j}$ under almost every $u\in U_+$.}
\end{figure}

\bold{Proof.} 


(a)\, Consider $x\in X$  such that (\ref{eq:Lyapunov:transversality}) holds both for $x$ and for almost every point in $\cB[x]$ and such that $x$ belongs to the set $Y$ described in 
Corollary~\ref{c.continuity}.  Since  (\ref{eq:Lyapunov:transversality})  holds almost everywhere, Fubini's theorem implies that the set of such $x$ has full measure in $X$. 
Pick such an $x$, and let $\bfv\in \bF_{ \geq j}(x)$.  Then we can write  $\bfv = \bfv_{\geq j} + \bfv_{\leq j-1}$, where
$\bfv_{\geq j} \in \bE_{\geq j}(x)$, and 
$\bfv_{\leq j-1} \in \bE_{\leq j-1}(x)$. We have $u_\ast \bfv = u_\ast  \bfv_{\geq j} + u_\ast  \bfv_{\leq j-1}$.  

If $\bfv_{\leq j-1}  \neq 0$, then  (\ref{eq:invariance:Vj:hatVj})  implies
that for  almost every $u\in U_+$, we have $0 \neq u_\ast \bfv_{\leq j-1}\in  \bE_{\leq j-1}(ux)$.

By our choice of $x$, we have $ \bE_{\leq j-1}(ux)\cap  \bE_{\geq
  j}(ux) = \emptyset$ for almost every $u\in U_+$.  Now suppose $u \in
\cG_x$, where $\cG_x\subset \cB$ given by
Corollary~\ref{c.continuity}.  
We decompose $u_* \bfv_{\geq j}$ as the sum of a vector in $\bE_{\geq j}(ux)$ and a vector in $\bE_{\leq j-1}(ux)$; since both of these spaces vary continuously in $u\in \cG_x$, if $u\in \cG_x$ is sufficently close to the identity
in $U_+$, then  the component of   $u_* \bfv_{\geq j}$ in $\bE_{\leq j-1}(ux)$ is arbitrarily small, and in particular smaller than $u_* \bfv_{\leq j-1}$. 

This shows that $u_*\bfv \notin  \bE_{\geq j}(ux)$ for $u\in \cG_x$ sufficiently close to the identity.  But Corollary~\ref{c.continuity} implies that the set of such $u$ has positive Lebesgue measure in $\cB$; this contradicts the assumption 
that $\bfv\in \bF_{ \geq j}(x)$.  Therefore $\bfv_{\leq j-1} = 0$, and  $\bfv\in \bE_{ \geq j}(x)$, proving (a).

(c)\, Note that for $t \geq 0$, 
$a^t \cB[x]  = \cB_{t}[a^t x] \supset \cB[a^t x]$. It follows that 
\begin{equation}\label{e.inclusion}
a^t_*(\bF_{\geq j}(x))\subset\bF_{\geq j}(a^t x ).
\end{equation}
In 
particular, the dimension of $\bF_{\geq j}(x)$ is a measurable function, with values in $\{0,\dots, \dim(H)\}$, 
and which is nondecreasing along  $a^t$-orbits. Thus the sets $\{x\in X : \dim(\bF_{\geq j}(x))\geq i\}$ are positively invariant by $a^t$, $t\geq 0$. 
As the measure $\nu$ is assumed to be $A$-ergodic, it follows that there exists $i\in\{0,\dots, \dim(H)\}$ such that, for $\nu$-almost every $x$, the dimension of 
$\bF_{\geq j}(x)$ is equal to $i$.  For such a $\nu$-generic $x$, the inclusion in (\ref{e.inclusion}) becomes an equality:
\[
a^t_*(\bF_{\geq j}(x))=\bF_{\geq j}(a^t x),
\]
proving item (c).

(b)\, {Consider now an $x$ such that (c) holds,}
a vector $\bfv\in
\bF_{\geq j}(x)$, and  a large $t>0$.  Now $a^{-t}_* \bfv$ belongs
to $\bF_{\geq j}(a^{-t}x)$. Since $a^t \cB[a^{-t}x]  = \cB_{t}[x]$, we deduce that for every $t>0$, 
\[\bF_{\geq j}(x) = \{ \bfv \in \bH(x) \st \text{ for almost 
all $u \in \cB_{t}(x)$, $u_* \bfv \in \bE_{\geq j}(u x)$} \}.\]
As $t>0$ was arbitrary, this establishes (b).

(d)\, Consider $x$ satisfying (a) and (b), and let $\bfv\in \bF_{\geq
  j}(x)$.  Then for almost every $u\in U_+$, we have $u_\ast \bfv\in
\bE_{\geq j}(ux)$.  Fix such a $u$, and suppose that  $u_\ast \bfv \notin \bF_{\geq j}(ux)$. 
This implies there there exists a positive measure set of $u'\in U_+$ such that $u'u_\ast \bfv \notin \bE_{\geq j}(u'ux)$.  
But this violates the assumption that $\bfv\in \bF_{\geq j}(x)$.  Thus for almost all $u\in U_+$, we have 
$u_* \left(\bF_{\geq j}(x)\right) \subset \bF_{\geq j}(ux)$.  The reverse inclusion is proved by applying $u^{-1}$.  This gives (d).
\qed\medskip

%
\subsection{The inert flag is \(G\)-invariant}
\begin{lemma}
\label{lemma:Fj:invariant:Uminus}
For almost all $\bar{u} \in U_-$ and almost all $x \in X$, $\bar{u}_*
\left(\bF_{\geq j}(x) \right)= \bF_{\geq j}(\bar{u} x)$. 
\end{lemma}

\bold{Proof.} It is enough to show that (outside a set of measure
$0$), for all $\bfv \in \bF_{\geq j}( x)$,  we have $\bar{u}_* \bfv \in
\bF_{\geq j}(\bar u x)$. Hence, in view of the definition (\ref{eq:def:bFj}) 
of $\bF_{\geq j}$, it is enough
to show that for almost all $u \in U_+$ and $\bar{u} \in U_-$, 
\begin{equation}
\label{eq:Fj:invariant:need:to:prove}
(u \bar{u})_* \bfv \in  \bE_{\geq j}( u \bar{u} x). 
\end{equation}
We may write $u \bar{u} = \bar{u}' u' a$ where $\bar{u}' \in U_-$,
$u' \in U_+$, $a \in A$. Let
\begin{displaymath}
\bfw = (u' a)_* \bfv. 
\end{displaymath}
Then, by Lemma~\ref{lemma:Ej:equivariant} (c) and (d), 
$\bfw \in \bF_{\geq j}(u' a x)$. By
Lemma~\ref{lemma:Ej:equivariant} (a), $\bF_{\geq j}(u' a x) \subset
\bE_{\geq j}(u' a x)$. Therefore
\begin{displaymath}
(u \bar{u})_* \bfv = (\bar{u}' u' a)_* \bfv = \bar{u}'_* \bfw \in
\bar{u}'_*\left( \bF_j(u' a x)\right) \subset \bar{u}'_* \left(\bE_{\geq j}(u' a x)\right).
\end{displaymath}
But, by (\ref{eq:invariance:Vj:hatVj}),
\begin{displaymath}
\bar{u}'_* \left(\bE_{\geq j}(u' a x) \right) = \bE_{\geq j}(\bar{u}' u'
a x) = \bE_{\geq j}(u \bar{u} x). 
\end{displaymath}
Thus, (\ref{eq:Fj:invariant:need:to:prove}) holds. 
\qed\medskip

In view of Lemma~\ref{lemma:Ej:equivariant} (c),  (d) and
Lemma~\ref{lemma:Fj:invariant:Uminus}, we have proved the following:
\begin{proposition}
\label{prop:Fj:invariant:SL2R}
The subspaces $\bF_{\geq j}(x)$ are equivariant under the action of
$SL(2,\reals)$. 
\end{proposition}

\subsection{Relaxing the definition of the components of the inert flag.}

\begin{lemma}
\label{lemma:zero:one:substitute}
There exists a subset $\Omega \subset X$ with $\nu(\Omega) =
1$ with the following property.  For $j\in \{0,\ldots, n\}$, $x \in \Omega$, and $\bfv \in
\bH(x)$, let
\begin{equation}
\label{eq:zero:one:substitute}
Q_j(\bfv) = \{ u \in \cB \st u_* \bfv \in \bE_{\geq j}(u x)\}. 
\end{equation}
Then either $|Q_j(\bfv)| = 0$, or $|Q_j(\bfv)| = |\cB|$ (and thus
$\bfv \in \bF_j(x)$). 
\end{lemma}
\bold{Proof.} Fix   $j\in \{0,\ldots, n\}$.  For a subspace $\bV \subset \bH(x)$, let 
\begin{displaymath}
Q_j(\bV) = \{ u \in \cB \st  u_* \bV  \subset \bE_{\geq j}(u x)\}. 
\end{displaymath}
Let $d$ be the maximal number such that there exists 
$Y\subset X$ with $\nu(Y) > 0$ such that for $x \in Y$ there exists a
subspace $\bV \subset \bH(x)$ 
of dimension $d$ with $|Q_j(\bV)| > 0$. 
For a fixed $x \in Y$, let
$\cW(x)$ denote the set of subspaces $\bV$ of dimension $d$ 
for which $|Q_j(\bV)| > 0$. Then, by
the maximality of $d$, if $\bV$ and $\bV'$ are distinct elements of
$\cW(x)$ then $Q_j(\bV) \cap Q_j(\bV') = \emptyset$.  Fix a measurable
collection of subspaces $\bV_x \in \cW(x)$, for $x\in Y$, 
such that $|Q_j(\bV_x)|$ is maximal (among elements of $\cW(x)$).

For $x\in Y$,  $u\in \cB$ and $t>0$, consider the set 
\[D(u,x):=\{ z \in \cB_{-t}  \st z u  \in Q_j(\bV_x)\}.
\]
Let $\epsilon > 0$ be arbitrary, and suppose $x \in Y$.
The Vitali covering lemma  implies that there exists $t_0 >0$ and 
a subset $Q(\bV_x)^* \subset Q(\bV_x) \subset \cB$ such that
\begin{equation}
\label{eq:ux:point:of:density0}
|Q(\bV_x) \setminus Q(\bV_x)^\ast| < \epsilon|Q(\bV_x)|,
\end{equation}
and for all $u \in Q(\bV_x)^*$ and
all $t > t_0$, we have
\begin{equation}
\label{eq:ux:point:of:density}
|D(u,x)| \ge (1-\epsilon) |\cB_{-t}|.   
\end{equation}
(In other words, $Q(\bV_x)^*$ are ``points of $(1-\epsilon)$ density'' for $Q(\bV_x)$,
relative to the  collection of (small) balls $\{\cB_{-t}u \st u\in Q(\bV_x), t>t_0\}$.) 
Let 
\begin{displaymath}
Y^* = \{ u x \st x \in Y, \quad u \in Q(\bV_x)^* \}. 
\end{displaymath}
Then, since $\nu(Y)>0$, the $U_+$-invariance of $\nu$ and (\ref{eq:ux:point:of:density0}) imply that   $\nu(Y^*) > 0$. 
Let $\Omega = \{ x \in X \st a^{-t} x \in Y^*
\hbox{ infinitely often.}\}$. Poincar\'e recurrence implies that $\nu(\Omega) = 1$. Suppose $x \in
\Omega$. We can choose $t >t_0$ such that $a^{-t} x \in
Y^*$.
Note that
\begin{equation}
\label{eq:relation:balls}
\cB[x] = a^t \cB_{-t} [a^{-t}x].
\end{equation}
\mcc{Better to work with $\cB[x]$
  instead of $\cB$}
Let $x' = a^{-t}x \in Y^\ast$, and let $\bV_{t,x} = (a^t)_*
\bV_{x'}$. Then in view of (\ref{eq:ux:point:of:density}) and
(\ref{eq:relation:balls}),  we  have
\begin{equation}
\label{eq:QV:1:minus:epsilon}
|Q(\bV_{t,x})| \ge (1-\epsilon)|\cB| .
\end{equation}
By the maximality of $d$ (and assuming $\epsilon <
  1/2$), $\bV_{t,x}$ does not depend on $t$. 
Hence, for every $x \in \Omega$, there exists $\bV \subset \bH(x)$ such that
$\dim \bV = d$ and $|Q(\bV)| \ge (1-\epsilon) |\cB|$. Since $\epsilon
> 0$ is arbitrary, for each $x \in \Omega$, there exists $\bV \subset
\bH(x)$ with $\dim \bV =d$, and $|Q(\bV)| = |\cB|$. 
Now the maximality of
$d$ implies that if $\bfv \not\in \bV$ then $|Q(\bfv)| = 0$. 
\qed\medskip

\section{Proof of Theorem~\ref{theorem:P:uniq:ergodicity}}
\label{sec:proof:of:uniq:ergodicity}

Now suppose that there is more than one Lyapunov exponent on $\bH$ 
(so that $\lambda_2 < \lambda_1$). Then for almost all $x$, 
$\bF_{\geq 2}(x)$ is a proper, nontrivial subspace of $\bF_{\geq 1}(x) = \bH(x)$. 

Suppose also that the cocycle is irreducible with respect to the measure $\nu$ (so there are no
non-trivial proper equivariant $\nu$-measurable subbundles). Then it follows from
Proposition~\ref{prop:Fj:invariant:SL2R} that for all $j \ge 2$ and 
$\nu$-almost all $x$, 
\begin{equation}
\label{eq:bFj:zero}
\bF_{\geq j}(x) = \{ 0 \}. 
\end{equation}

Let $\hat{\nu}\in \cM_P(\nu)$ be any $P$-invariant measure on the total space of
the projectivized bundle $\proj^1(\bH)$ that projects to $\nu$. 
We may disintegrate $\hat{\nu}$ as follows:
\begin{displaymath}
d \hat{\nu}([\bfv]) = d \nu(\pi(\bfv)) \, d\eta_{\pi(\bfv)}( [\bfv]).
\end{displaymath}

 Lemma~\ref{lemma:zero:one:substitute}   and (\ref{eq:bFj:zero}) imply that for almost every $x\in X$
and every $\bfv\in \bH(x)\setminus \{0\}$:
\begin{equation}
|\{ u \in \cB \st u_* \bfv \in \bE_{\geq 2}(u x) \}| = 0. 
\end{equation}
In other words, for almost every $x$ and every nonzero $\bfv\in \bH(x)$, we have 
$u_\ast\bfv \notin \bE_{\geq 2}(ux)$, for almost every $u\in \cB$.  

The $U_+$-invariance of $\hat\nu$ then implies the following claim.

\begin{claim}
For $\nu$-almost all $x$, the disintegration $\eta_x$ of $\hat \nu$ is
supported in the set 
$\proj\left(\bH(x)\setminus  \bE_{\geq 2}(x)\right)$. 
\end{claim}
\begin{proof}First notice that almost all $P$-invariant measures in the $P$-ergodic desintegration of $\hat \nu$ project on $\nu$. 
Therefore it is suffices to prove the claim under the extra assumtion that $\hat \nu$ is ergodic for $P$. 

We argue by contradiction, assuming that there is a measurable set $Z$ contained in $\proj(\bE_{\geq 2})$  and  a positive $\nu$-measure set $Y\subset X$ such  that $\eta_y(Z(y))>0$ for $y\in Y$, where $Z(y)= Z\cap \proj(\bH(y))$. 
In other words $\hat\nu(Z)>0$.  As the bundle $\bE_{\geq 2}$ and the disintegrations $\eta_y$ are equivariant under the action
of $\{a^t\}$, we may assume that $Y$ and $Z$ are invariant under the action of $a^t$.

Since $P$ is amenable, and $\hat \nu$ is ergodic with respect to $P$, the Mean Ergodic Theorem for amenable actions (see  \cite[Theorem 8.13]{EinsiedlerWard}) implies that for $\hat\nu$-almost every point $[\bfv]$ in $\proj(\bH)$ there is 
a set  $\cC_{[\bfv]}\subset P$ of positive measure with respect to Haar measure on  $P$,  such that for
$h\in \cC_{[\bfv]}$, we have $h_*([\bfv])\in Z$, 
and therefore $h_*([\bfv])\in \proj(\bE_{\geq 2}(hx))$,  
where $x\in X$ is the base point of $[\bfv]$. 

As $Z$ is invariant under the action of $a^t$, the set $\cC_{[\bfv]}$ is invariant under the action of $a^t$ on the group $P$; it follows that
$\cC_{[\bfv]}$ intersects ${U_+}$ in a set of positive Lebesgue measure.  According to Lemma~\ref{lemma:zero:one:substitute}, this implies
that $[\bfv]\in \bF_{\geq 2}(x)$.  In particular, one gets that $\bF_{\geq 2}(x)$ is not trivial for $\nu$-almost every point $x\in X$, contradicting 
 the assumption of $\nu$-irreducibility. 
\end{proof}

Using the $A$-invariance of $\hat\nu$  and Oseledets' theorem,  we obtain the next claim, which gives the main conclusion of Theorem~\ref{theorem:P:uniq:ergodicity}:
\begin{claim}$\eta_x$ must in fact be supported in $\proj({\bE_1}(x))$ for
$\eta$-a.e.\ $x$.
\end{claim}
\begin{proof} To any point $[\bfv]$ of $\proj\bH(x)\setminus \proj(\bE_{\geq 2}(x))$ at a point $x$ generic for $\nu$ for which the Lyapunov spaces are well defined, 
we associate a real number $\alpha([\bfv])$ as follows: there is a vector $\bfv\in [\bfv]$ so that $\bfv=\bfv_1+\alpha([\bfv]) \bfv_{\geq 2}$ where $\bfv_1\in \bE_1(x)$ and 
$\bfv_{\geq 2}\in \bE_{\geq 2}(x)$ are unit vectors. 

As $a^t$ preserves the Lyapunov spaces, the set where $\alpha$ is defined is invariant under the action of $a^t$ on $\proj \bH$. Furthermore,
by definition of the Lyapunov spaces one gets
$$\lim_{t\to +\infty}\alpha( (a^t)_*([\bfv]))= 0,$$
for any $[\bfv]\in \proj\bH(x)\setminus \proj(\bE_{\geq 2}(x))$, where  $x$ is generic for $\nu$. 

The previous claim said that $\alpha$ is well defined $\hat\nu$-almost everywhere. 
We want to prove that $\alpha$ vanishes $\hat\nu$-almost everywhere. 

Otherwise, there is a compact interval $[a,b]\subset (0,+\infty)$ so that $\hat\nu(Z_{a,b})>0$ where $Z_{a,b}=\{(x,[\bfv]): \alpha([\bfv])\in[a,b]\}$. 
Thus $Z_{a,b}$ is a set of positive measure but every point in $Z_{a,b}$ has only finitely many positive return times for $a^t$ in $Z_{a,b}$, 
contradicting the Poincar\'e recurrence theorem. This contradiction concludes the proof. \end{proof}

If, moreover, the top Lyapunov exponent
$\lambda_1$ is simple, it follows that for almost all $x$, $\eta_x$
must be the Dirac measure supported on $\proj(\bE_1(x))$. Hence, in this
case, there is
only one invariant measure.\qed\medskip

\section{Proof of Theorem~\ref{theorem:convergence:of:Lyapunov:exponents}}
\label{sec:proof:of:theorem:convergence}
{ 
We begin with the general observation (due to Furstenberg) that Lyapunov exponents of a cocycle on $\bH$ can be computed explicitly as integrals over $\proj(\bH)$.

 Let  $\pi\colon \bH\to X$ be a bundle  with an $A=\{a^t\}$ action,
and let $\|\cdot \|_x$ be a Finsler on $\bH$.   Let $\nu$ be an ergodic $A$-invariant probability measure
on $X$ satisying the integrability condition
 (\ref{eq:cocycle:integrability}).  Define 
$\sigma\colon A\times \proj(\bH) \to \reals$ by
\begin{equation}\label{e=sigmadef}\sigma(a^t, [\bfv] ) = \log\frac{\|a^t_\star(\bfv)\|_{g(\pi(\bf v))}}{\|\bfv\|_{\pi(\bf v)}},
\end{equation}
where $\bfv$ is any nonzero vector representing the projective class $[\bf v] := {\mathbb F} \bf v \in \proj(\bH)$ (where ${\mathbb F}$ is the base field).
Note that $\sigma$ is a real-valued cocycle over the action of $A$ on $\proj({\bH})$.  By (\ref{eq:cocycle:integrability}), this cocycle
is integrable with respect to {\em any}  probability measure $\hat\nu$ on $\proj(\bH)$ projecting to $\nu$.

Denote by $\lambda_1(\nu)$ the top Lyapunov exponent for $a^t$ with respect to $\nu$, as in 
(\ref{eq:osceledts:two:sided:splitting}). Let
\begin{displaymath}
Z_1(\nu) = \{ \bfv \in \proj^1(\bf H) \st \bfv \in \proj(\bE_1(\pi(v)))\},
\end{displaymath}
where $\bE_1$ is as in (\ref{eq:osceledts:two:sided:splitting}). 

We
recall the following elementary result:
\begin{lemma}
\label{lemma:furstenberg:formula}
Let $\hat{\nu}$ be any $a^t$-invariant measure on $\proj(\bH)$ that
projects to $\nu$ and is supported in $Z_1(\nu)$. Then
\begin{displaymath}
\lambda_1(\nu) = \int_{ \proj^1(\bf H)} 
\sigma(a^1, [\bfv]) \, d\hat{\nu}([\bfv]).
\end{displaymath}
(Here $a^1$ means $a^t$ for $t=1$).
\end{lemma}

\bold{Proof.} Suppose $t \in \natls$. Since $\sigma$ is a cocycle,
\begin{displaymath}
\sigma(a^t, [\bfv]) = \sum_{n=1}^t \sigma(a^1, [a^{t-n}_\ast\bfv])). 
\end{displaymath}
Since $\hat{\nu}$ is $a^t$-invariant, integrating both sides over $\proj^1(\bH)$ with respect to $\hat{\nu}$ and dividing both
sides by $t$, we get
\begin{multline}
\label{eq:tmp:integration}
\frac{1}{t} \int_{\proj^1(\bH)} \sigma(a^t,  [\bfv]) \,
d\hat{\nu}([\bfv]) = \frac{1}{t} \sum_{n=1}^t \int_{\proj^1(\bH)} \sigma(a^1, [\bfv]) \, d \hat{\nu}( [a^{t-n}_\ast\bfv]) = \\ = \int_{\proj^1(\bH)}  \sigma(a^1,[\bfv]) \, d \hat{\nu}([\bfv]).
\end{multline}
However, in view of the assumption that $\hat{\nu}$ is supported in
$Z_1$,  it follows from ergodicity of $\nu$ and the
multiplicative ergodic theorem that the left-hand side of
(\ref{eq:tmp:integration}) tends to $\lambda_1(\nu)$ as $t \to
\infty$. 
\qed\medskip

\bold{Proof of Theorem~\ref{theorem:convergence:of:Lyapunov:exponents}.}

{  Let $\cN_n$, $\cN$, $\nu_{\cN_n}$, $\nu_{\cN}$ be as in
Theorem~\ref{theorem:convergence:of:Lyapunov:exponents}. 
Fix the Finsler structure $\|\cdot\|_x$ on the Hodge bundle  over $ \cH_1(\beta)$ so that the 
Kontsevich-Zorich cocycle satisfies the uniform integrability condition in (\ref{eq:cocycle:uniformintegrability}).}

By \cite[Theorem~A.6]{EM} (which is essentially due to Forni
\cite{Forni:Deviation}), the Kontsevich-Zorich cocycle restricted to
the affine manifold $\cN$ is semisimple. This means that 
after passing to a finite cover, we have the direct sum decomposition
\begin{equation}
\label{eq:decomp:H1}
H^1(M,\reals) = \bigoplus_{i=1}^m W_i(x)
\end{equation}
where the $W_i$ are $G$-equivariant, and the restriction of the
cocycle to each $W_i$ is strongly irreducible. The map $x \mapsto W_i(x)$
is {\em a priori} only measurable, but 
Theorem~\ref{theorem:filip} implies that it is continuous (and in fact
real analytic). 

Note that by
Theorem~\ref{theorem:mozes-shah}, we have $\cN_n \subset \cN$ for
sufficiently large $n$. Then, for sufficiently large $n$, the
decomposition (\ref{eq:decomp:H1}) also holds for each $\cN_n$. 

It is clearly enough to prove the theorem for the restriction of the
cocycle to each $W_i$. We thus let $\bH(x) = W_i(x)$ (for some fixed
$i$), and we may now assume that the restriction $\alpha_\bH: G \cross \bH\to \bH$ of the
Kontsevich-Zorich cocycle to $\bH$ is strongly irreducible on $\cN$.  

Let $\lambda_1(\cN_n)$ and $\lambda_1(\cN)$ denote the top
Lyapunov exponents of $\alpha_\bH$ with respect to the affine measures
$\nu_{\cN_n}$ and $\nu_\cN$. 

Note that for $\nu_{\cN_n}$-a.e. $x \in X$ the set
$\proj (\bE_1(x)) \subset \proj(\bH)(x)$ is closed, and the set 
$Z_1(\nu_{\cN_n})$ is $P$-invariant. Then
by the amenability of $P$, for each $n$ there exists
a $P$-invariant measure $\hat{\nu}_n$ on $\proj^1(\bH)$
such that $\hat{\nu}_n$ projects to $\nu_{\cN_n}$ and is supported in 
$Z_1(\nu_{\cN_n})$. By Lemma~\ref{lemma:furstenberg:formula}, 
\begin{equation}
\label{eq:lambda1:hat:nu:n}
\lambda_1(\cN_n) = \int_{ \proj^1(\bH)} \sigma(a^1, [\bfv])
\, d\hat{\nu}_n([\bfv]), 
\end{equation}
where 
$\sigma\colon A\times \proj(\bH) \to \reals$ is the cocycle defined by (\ref{e=sigmadef}).

Let $\hat{\nu}$ be any weak-star limit of the measures
$\hat{\nu}_n$. Then $\hat{\nu}$ is a $P$-invariant measure that
projects to $\nu_\cN$. By Theorem~\ref{theorem:P:uniq:ergodicity},
(each ergodic component of) 
$\hat{\nu}$ is supported in $Z_1(\nu_\cN)$.

 Therefore, by
Lemma~\ref{lemma:furstenberg:formula}, 
\begin{equation}
\label{eq:lambda1:hat:nu}
\lambda_1(\cN) = \int_{\proj^1(\bH)} \sigma(a^1, [\bfv])
\, d\hat{\nu}([\bfv]). 
\end{equation}

{  Although $X =  \cH_1(\beta)$ is not compact,
it follows from \cite{Eskin:Masur}
that for every $\delta > 0$ there exists a
compact set $K_\delta \subset  \cH_1(\beta)$ such that for every
$G$-invariant measure $\nu$, we have $\nu(K_\delta) > 1-\delta$.
Given $\epsilon>0$, let $\delta>0$ be given by the uniform integrability assumption (\ref{eq:cocycle:uniformintegrability}), and fix the compact set $K_\delta$.

Consider the restriction  $\proj_{K_\delta}(\bH)$  of  $\proj(\bH)$  to this compact set.
The weak convergence $ \hat{\nu}_n   \to  \hat{\nu}$ implies that the integral of the continuous, compactly supported function ${\bf 1}_{\proj_{K_\delta}(\bH)} \cdot \sigma$ with respect to $\hat\nu_n$ converges to the integral with respect to $\hat\nu$, and the integral of ${\bf 1}_{\proj{\bH} \,\setminus\, \proj_{K_\delta}(\bH))} \cdot \sigma$ with respect to all $\hat\nu_n$ and $\hat\nu$ is  less than $\epsilon$, by the uniform integrability assumption (\ref{eq:cocycle:uniformintegrability}).  Since $\epsilon>0$ is arbitrary, we conclude that the integrals in 
(\ref{eq:lambda1:hat:nu}) converge as $n\to\infty$ to the integral in  (\ref{eq:lambda1:hat:nu:n}), and so
 $\lambda_1(\cN_n) \to \lambda_1(\cN)$. }

To show convergence of other Lyapunov exponents, it suffices to repeat
the argument for the cocycle acting on exterior powers of $H^1(M,
\reals)$. We note that the cocycle remains semisimple in this setting
and the analogue of (\ref{eq:decomp:H1}) still holds (see
\cite{Filip}).
\qed

\end{document}